\documentclass[12pt]{amsart}
\usepackage{amsmath,amssymb}
\usepackage[margin=1in]{geometry}
\theoremstyle{plain}
\newtheorem{thm}{Theorem}[section]

\newtheorem{prop}[thm]{Proposition}
\newtheorem{remark}{Remark}  

\newtheorem{corollary}[thm]{Corollary}

\renewcommand{\div}{\operatorname{div}}
\newcommand{\D}{\displaystyle}
\newcommand{\ip}[2]{\ensuremath{\langle #1 , #2 \rangle}}
\newcommand{\tp}{\texttt{p}}
\newcommand{\spn}{\operatorname{span}}

\newcommand{\no}{\nabla_0}
\renewcommand{\cap}{\operatorname{cap}}
\newcommand{\eucl}{\operatorname{eucl}}

\newcommand{\css}{\operatorname{\mathcal{C}^2_{\textmd{sub}}}}
\newcommand{\cs}{\operatorname{\mathcal{C}^1_{\textmd{sub}}}}

\theoremstyle{definition}
\newtheorem{definition}{Definition}

\theoremstyle{remark}

\numberwithin{equation}{section}
\begin{document}
\title[Fundamental Solution in H\"{o}rmander Vector Fields]{The Fundamental Solution to the $\tp$-Laplacian in a class of H\"{o}rmander Vector Fields}
\author{Thomas Bieske}
\address{ Department of Mathematics and Statistics,
University of South Florida, Tampa, FL 33620, USA}
\email{tbieske@mail.usf.edu}
\thanks{This paper is part of the second author's Ph.D. thesis under the direction of the first author.}
\author{Robert D. Freeman}
\address{ Department of Mathematics and Statistics,
University of South Florida, Tampa, FL 33620, USA}
\email{rfreeman1@mail.usf.edu}

\begin{abstract} 
We find the fundamental solution to the $p$-Laplace equation in a class of H\"{o}rmander vector fields that generate neither a Carnot group nor a Grushin-type space.  The singularity occurs at the sub-Riemannian points, which naturally 
corresponds to finding the fundamental solution of a generalized operator in Euclidean space. We then use this solution to find an infinite harmonic function with specific boundary data and to compute the capacity of annuli centered at the singularity. 
\end{abstract}
\maketitle
\section{Motivation}
The $\texttt{p}$-Laplace equation is the model equation for nonlinear potential theory.  The Euclidean results of \cite{HKM:NPT} can be extended into a class of sub-Riemannian spaces possessing an algebraic group law, called Carnot groups \cite{HH:QR}. Fundamental solutions to the $p$-Laplace equation in a subclass of Carnot groups called groups of Heisenberg-type have been found in \cite{CDG, HH:QR}. The exploration of the $p$-Laplace equation in sub-Riemannian spaces without an algebraic group law is currently a topic of interest. In H\"{o}rmander vector fields without an algebraic group law, fundamental solutions to the $\texttt{p}$-Laplace equation have been found in a subclass called Grushin-type spaces. \cite{BG, B, WB} In this paper, we find the fundamental solution to the $\tp$-Laplace equation for $1< \tp < \infty$ in a class of H\"{o}rmander vector fields that are not Grushin-type. The singularity occurs at the sub-Riemannian points, which naturally corresponds to finding the fundamental solution of a generalized operator in Euclidean space. To our knowledge, this is the first instance of a closed-form fundamental solution to the $\texttt{p}$-Laplace equation outside of Carnot groups and Grushin-type spaces. 
\section{The Environment}
Consider $\mathbb{R}^{2n+1}$ and fix a point $x_0=(a_1,a_2,\ldots, a_{2n}, s)\in\mathbb{R}^{2n+1}$ and numbers $0\neq c \in\mathbb{R}$ and $k\in\mathbb{R}^+$.
We begin with the following vector fields:

\begin{eqnarray}\label{vectorfields}
X_i =\left\{\begin{array}{cc} 
\frac{\D\partial}{\D\partial{x_i}}+ 2kc(x_{i+n} -a_{i+n})\bigg(\D\sum_{l=1}^{2n}(x_l - a_l)^2\bigg)^{k-1}\frac{\D\partial}{\D\partial{t}}, & 1\leq i \leq n, \nonumber\\
\mbox{}\nonumber\\
 \frac{\D\partial}{\D\partial{x_i}} - 2kc(x_{i-n} -a_{i-n})\bigg(\D\sum_{l=1}^{2n}(x_l - a_l)^2\bigg)^{k-1}\frac{\D\partial}{\D\partial{t}}, & n+1 \leq i \leq 2n,
 \end{array}\right.
\end{eqnarray}
The Lie bracket $[X_i,X_j]$ for $i<j$, equals 
\begin{multline*}
8kc(k-1)\Big(\D\sum_{l=1}^{2n}(x_l - a_l)^2\Big)^{k-2}\bigg((x_{j+n}-a_{j+n})(x_i-a_i)-(x_{i+n}-a_{i+n})(x_j-a_j)\bigg)\frac{\partial}{\partial{t}} \textmd{\ \ for\ } i,j\leq n \\
8kc(k-1)\Big(\D\sum_{l=1}^{2n}(x_l - a_l)^2\Big)^{k-2}\bigg((x_{i-n}-a_{i-n})(x_j-a_j)-(x_{j-n}-a_{j-n})(x_i-a_i)\bigg)\frac{\partial}{\partial{t}} \textmd{\ \ for\ } i,j>n \\
8kc(k-1)\Big(\D\sum_{l=1}^{2n}(x_l - a_l)^2\Big)^{k-2}\bigg((x_{i+n}-a_{i+n})(x_j-a_j)-(x_{j-n}-a_{j-n})(x_i-a_i)\bigg)\frac{\partial}{\partial{t}}  \\
\mbox{}-4kc\Big(\D\sum_{l=1}^{2n}(x_l - a_l)^2\Big)^{k-1}\delta_{i,j-n}\frac{\partial}{\partial{t}}\textmd{\ \ for\ } i \leq n<j.
\end{multline*}
We note that this computation shows that these vector fields are not the tangent space of a Carnot group when $k\neq 1$, as the brackets all vanish at some points, such as $x_0$, but not at others. In addition, the vector fields do not possess the Grushin-type space form (see, for example, \cite{B, BG}). 

We endow $\mathbb{R}^{2n+1}$ with an inner product  so that the collection $\{X_i\}_{i=1}^{2n}\cup \{\frac{\partial}{\partial t}\}$ forms an 
orthonormal basis, producing a sub-Riemannian manifold that we 
shall call $g_n$ that also forms the tangent space to a sub-Riemannian space, denoted $G_n$. Points in $G_n$ will also be denoted by $x=(x_1,x_2,\ldots, x_{2n}, t)$.

Though $G_n$ is not a Carnot group, it is a metric space with the natural 
metric being the Carnot-Carath\'{e}odory distance, which is defined for 
points $x$ and $y$ as follows:
\begin{equation*}
d_C(x,y)= \inf_{\Gamma} \int_{0}^{1} \| \gamma '(t) \|\; dt. 
\end{equation*}
\noindent Here $ \Gamma $ is the set of all curves $ \gamma $ such 
that $ \gamma (0) = x, \gamma (1) = y $ and 
$$
\gamma '(t) \in \spn \{\{X_i(\gamma(t))\}_{i=1}^{2n}\} .
$$
Using this metric, we can define a Carnot-Carath\'{e}odory 
ball of radius $r$ centered at a point $x_0$ by
$$
B_C(x_0,r)=\{x\in G_n : d_C(x,x_0) < r\}
$$ 
and similarly, we shall denote a bounded domain in $G_n$ by $\Omega$. Given a smooth function $f$ on $G_n$, we define the horizontal gradient 
of $f$ as 
$$
\nabla_0f(x) = (X_1f(x),X_2f(x),\ldots, X_{2n}f(x))
$$
and the symmetrized second order (horizontal)  derivative matrix by 
$$
((D^2f(x))^{\star})_{ij} = \frac{1}{2} (X_iX_jf(x)+X_jX_if(x))
$$ 
for $i,j=1,2,\ldots n$.  

\begin{definition}
The function $f: G_n \to \mathbb{R}$ is said to be $\cs$ if $X_if$ is 
continuous for all $i=1,2,\ldots, 2n$.  Similarly, the function $f$ is $\css$ 
if $X_iX_jf(x)$ is continuous for all $i,j=1,2,\ldots, 2n$.
\end{definition} 

Using these derivatives, we consider two main operators on $\css$ functions 
called the $\tp$-Laplacian 
$$
\Delta_{\tp}f=\div(\|\nabla_0f\|^{\tp-2}\nabla_0f)=\sum_{i=1}^{2n}X_i
(\|\nabla_0f\|^{\tp-2}X_if)
$$ 
defined for $1 < \tp < \infty$ and the infinite Laplacian
\begin{eqnarray*}
\Delta_{\infty}f
   =  \sum_{i,j=1}^{2n} X_ifX_jfX_iX_jf 
   =   \ip{\nabla_0 f}{(D^2f)^{\star}\nabla_0 f}=\frac{1}{2}\nabla_0f \cdot \nabla_0\|\nabla_0f\|^2.    
\end{eqnarray*}

We may define Sobolev spaces in the natural way.  Namely, for any open 
set $\mathcal{O} \subset G_n$, the function $f$ is in the horizontal 
Sobolev space $W^{1,q}(\mathcal{O})$ if the functions $f,\: X_1f, \ldots, 
X_{2n}f$ lie in $ L^q(\mathcal{O}) $.  Replacing $ L^q (\mathcal{O})$ by 
$ L_{loc}^q(\mathcal{O})$, the space $ W_{loc}^{1,q}(\mathcal{O}) $ is 
defined similarly.  We may then use these Sobolev spaces to consider the 
above operators in the usual weak sense. 

\section{The Co-Area Formula and Measure Theory}

Let $\Omega \subset G_n$ be a bounded domain, and let $\psi \in \cs(\Omega)$ be a 
smooth, real-valued function which extends continuously to $\partial 
\Omega$.  For convenience, we write $\nabla$ for the Euclidean gradient 
on $G_n = \mathbb{R}^{2n+1}$.
In place of Fubini's Theorem for iterated integrals, we will make use of 
the following Co-Area Formula in the sub-Riemannian case via Theorem 4.2 in \cite{MSC:co}. 

\begin{thm}
Under the hypotheses as above, then for any function $g \in L^1(\Omega)$
\begin{equation} \label{coarea}
\iint_\Omega g \|\nabla \psi\| \, d\mathcal{L}_{2n+1} \; = \; 
\int_0^\infty \int_{\psi^{-1}\{r\}} g \, d\mathcal{H} dr,
\end{equation}
where $d\mathcal{L}_{2n+1}$ denotes Lebesgue $(2n+1)$-measure on $\Omega$, and $d\mathcal{H}$ denotes 
Hausdorff $(2n)$-measure on $\psi^{-1}(\{r\})$.
\end{thm}

\begin{remark}
As above, the theorem also holds for continuous functions $\psi$ which are 
smooth everywhere except at isolated points.
\end{remark}

\noindent
We now suppose a particular case, where $x_0 \in G_n$ has coordinates
$x_0 = (a_1, a_2, \ldots, a_{2n}, s)$
and $\psi$ is a non-negative radial function with $\psi(x_0) = 0$.  The following 
notation is suggestive for the inverse images of $\psi$.
$$
\begin{array}{ccccc}
B_R(x_0) &=& \psi^{-1}([0,R)) &=& \{ x \in \Omega : \psi(x) < R \} \\
\partial B_R(x_0) &=& \psi^{-1}(\{R\})&=& \{ x \in \Omega : \psi(x) = R \}
\end{array}
$$
The $x_0$ is omitted when it is clear from the context.
Now choose $g(x) := \|\nabla_0 \psi\|^\tp$.  Since $\|\nabla_0\psi\| \lesssim 
\|\nabla \psi\|$ we may apply the Co-Area Formula to the function
$g = (g / \|\nabla \psi\|)\cdot \|\nabla \psi\|$ to obtain the following
proposition.

\begin{prop}
Let $\mathcal{V}$ be an absolutely continuous measure to $\mathcal{L}_{2n+1}$ with Radon-Nikodym
derivative $g = [d\mathcal{V} / d\mathcal{L}_{2n+1}]$.  Then for sufficiently small $R > 0$,
\begin{equation} \label{coarea2}
\mathcal{V}(B_R) \; = \; \int_{B_R} d\mathcal{V} \; = \; 
\int_0^R \int_{\partial B_r} \frac{g}{\|\nabla \psi\|} \, d\mathcal{H} dr
\end{equation}
\end{prop}

In light of the equality in \eqref{coarea2}, we see that the measure space 
$(G_n, \mathcal{V})$ is globally Ahlfors $Q$-regular with respect to balls centered
at $x_0$.  In particular, for $R > 0$,
\begin{equation} \label{ahlfors}
\mathcal{V}(B_R) \; = \; \sigma_\tp R^Q
\end{equation}
where $Q = 2n+2k$ and $\sigma_\tp = \mathcal{V}(B_1)$ is a fixed positive 
constant.

For technical purposes we proceed to study the boundary behavior of 
precompact domains $\Omega$.  This now motivates the following definition.

\begin{definition}
For small values $R \in R_\psi$, define a measure $\mathcal{S}$ on $\partial B_R$ as
$$
\mathcal{S}(\partial B_R) \; = \; \int_{\partial B_R} d\mathcal{S} \; = \;
\int_{\partial B_R} \frac{g}{\|\nabla \psi\|} \,d\mathcal{H}.
$$ 
\end{definition}
In particular, $\mathcal{S}$ is absolutely continuous with respect to the Hausdorff 
$(2n)$-measure $\mathcal{H}$.  Using previous results, we now conclude:

\begin{corollary}
\begin{enumerate}
\item $\mathcal{S}$ is locally Ahlfors $(Q-1)$-regular and 
\begin{equation} \label{ahlfors2}
\mathcal{S}(\partial B_R) \; = \; Q \sigma_1 R^{Q-1}.
\end{equation}
\item Let $\varphi$ be a continuous and integrable function on $B_R$.  Then
as $R \to 0$,
\begin{equation} \label{density}
\frac{R^{1-Q}}{Q\sigma_\tp} \int_{\partial B_R} \varphi \,d\mathcal{S} \; \to \;
\varphi(0)
\end{equation}
\end{enumerate}
\end{corollary}

\begin{remark}
\eqref{ahlfors2} follows immediately from differentiating both \eqref{coarea2}
and \eqref{ahlfors}.  Since $\mathcal{S}$ is absolutely continuous with respect to 
Hausdorff $(2n)$-measure $\mathcal{H}$, it follows that $\mathcal{S}$ is Borel regular.  As a result,
\eqref{density} is the analogue of the Lebesgue Density Theorem. 
\end{remark}

\section{The $\tp$-Laplace Equation}
Now, we compute an explicit formula for the fundamental solution of the $\tp$-Laplacian for the vector fields defined by Equation \eqref{vectorfields} above and for $1 < \tp < \infty$. We define the constant $Q=2k+2n$. 

\begin{thm}\label{fund_sol}
Let $x_0 =(a_1, a_2, \ldots,  a_{2n},s)$ be an arbitrary fixed point.  Let $Q=2n+2k$.  Consider the following quantities for $1 < \tp < \infty$:
\begin{eqnarray*}
w & = & \frac{Q-\tp}{(1-\tp)(4k)}, \ \ \alpha =  \frac{Q-\tp}{(1-\tp)}, \\ 
h(x_1, \ldots,  x_{2n}, t)  & = &  c^2\bigg(\sum_{l=1}^{2n}(x_l - a_l)^2\bigg)^{2k} + (t-s)^2
\equiv c^2\Sigma^{2k}+ (t-s)^2, \\
\psi(x_1, \ldots,  x_{2n}, t)  & = &  [h(x_1, \ldots,  x_{2n}, t)]^{\frac{1}{4k}}, \\ 
f(x_1, \ldots, x_{2n}, t)  & = &  [h(x_1, \ldots,  x_{2n}, t)]^w = \psi(x_1, \ldots,  x_{2n}, t)^\alpha\\ 
\sigma_{\tp}  & = &  \int_{B_1} \|\no \psi \|^{\tp}\; d\mathcal{L}_{2n+1},\\ & &  \textnormal{where} \;d\mathcal{L}_{2n+1} \; \textnormal{denotes the Lebesgue $(2n+1)$-measure} \\
C_1  & = &  \alpha^{-1}(Q\sigma_{\tp})^{\frac{1}{1-\tp}}, \\ 
C_2  & = &  (Q\sigma_{\tp})^{\frac{1}{1-\tp}}.
\end{eqnarray*}
Then, for constants $C_1$ and $C_2$, we have
\begin{eqnarray}
\Delta_{\tp} C_1 f(x_1, ..., x_{2n}, t)= \delta_{x_0} \: \: when \: \: \tp \ne Q, \\
\Delta_{\tp} (C_2 \log \psi(x_1, ..., x_{2n}, t))= \delta_{x_0} \: \: when \: \: \tp = Q,
\end{eqnarray}
in the sense of distributions.
\end{thm}
\begin{proof}
Note that, for the sake of rigor, we should invoke the regularization of $h$ given by
$$h_\varepsilon(x_1, ..., x_{2n},t) = c^2\bigg(\sum_{l=1}^{2n}(x_l - a_l)^2 + \varepsilon^2\bigg)^{2k} + (t-s)^2$$
for $\varepsilon > 0$ and let $\varepsilon \to 0$.  Instead, we proceed formally.  We will need some calculations for the proof that we will compute first.  For $\tp \ne Q$, we have:
\begin{eqnarray*}
\textmd{\ \ when \ \ }i\leq n,\  X_i f & = & \alpha h^{w-1}c^2\Sigma^{2k-1}(x_i-a_i) + \alpha ch^{w-1} \Sigma^{k-1} (x_{i+n} - a_{i+n}) (t-s)  \\
\textmd{and when \ \ }j> n,\  X_j f & = & \alpha h^{w-1}c^2\Sigma^{2k-1}(x_j-a_j) - \alpha ch^{w-1} \Sigma^{k-1} (x_{j-n} - a_{j-n}) (t-s), \\
\textmd{so that\ \ } \|\no f \|^2 & = & \alpha^2c^2h^{2w-1}\Sigma^{2k-1}\\
\textmd{and\ \ } \|\no f \|^{\tp-2}  & = & |c\alpha|^{p-2}h^{(w-\frac{1}{2})(\tp-2)}\Sigma^{(k-\frac{1}{2})(\tp-2)}.\end{eqnarray*}
We then have for $i\leq n$, 
\begin{eqnarray*}
\|\no f \|^{\tp-2} X_i f   & = &  \alpha|\alpha|^{p-2}|c|^{\tp}h^{(w(\tp-1)-\frac{\tp}{2})}\Sigma^{(k\tp-\frac{\tp}{2})}(x_i-a_i)\\
& & \mbox{}+\alpha|\alpha|^{p-2}c|c|^{\tp-2}h^{(w(\tp-1)-\frac{\tp}{2})}\Sigma^{(k(\tp-1)-\frac{\tp}{2})}(x_{i+n}-a_{i+n})(t-s)
\end{eqnarray*}
and for $j>n$, we have 
\begin{eqnarray*}
\|\no f \|^{\tp-2} X_j f   & = &  \alpha|\alpha|^{p-2}|c|^{\tp}h^{(w(\tp-1)-\frac{\tp}{2})}\Sigma^{(k\tp-\frac{\tp}{2})}(x_j-a_j)\\
& & \mbox{}-\alpha|\alpha|^{p-2}c|c|^{\tp-2}h^{(w(\tp-1)-\frac{\tp}{2})}\Sigma^{(k(\tp-1)-\frac{\tp}{2})}(x_{j-n}-a_{j-n})(t-s).
\end{eqnarray*}

Letting $\Upsilon=w(\tp-1)-\frac{\tp}{2}$ and $\chi=k\tp-\frac{\tp}{2}$, we employ routine calculations to compute the $\tp$-Laplacian:
\begin{eqnarray*}
(\alpha|\alpha|^{p-2}|c|^{\tp})^{-1}\Delta_{\tp}f & = & 
 (\alpha|\alpha|^{p-2}|c|^{\tp})^{-1}\bigg(\sum_{i=1}^{n}X_i(\|\no f\|^{\tp -2} \no f) +  \sum_{j=n+1}^{2n}X_j(\|\no f\|^{\tp -2} \no f)\bigg)\\
  & = & (2n+2k+2\chi+ 4k\Upsilon) h^\Upsilon \Sigma^\chi \\
  & = & Q+2k\tp-\tp +(\tp-Q)-2k\tp = 0. 
\end{eqnarray*}

Note that these computations are valid wherever the function $f$ is 
smooth and in particular, these are valid away from the point $x_0$. 
We note that by our computations above,  $\|\nabla_0f\|^{\tp-1}$ is locally integrable on $G_n$.
We then consider $\phi \in C^{\infty}_0$ with compact support in the ball 
$$
B_R = \{y: \psi(y) < R \}.
$$   
Let $0 < r < R$ be given so that $B_r \subset B_R$.  In the annulus 
$\mathcal{A} := B_R \setminus \overline{B_r}$ we have, via the Leibniz rule,
\begin{eqnarray*}
\div(\phi \|\nabla_0f\|^{\tp-2}\nabla_0f) & = & 
\phi \div(\|\nabla_0f\|^{\tp-2}\nabla_0f) +
\|\nabla_0f\|^{\tp-2}\ip{\nabla_0f}{\nabla_0\phi} \\ & = & 
0 + \|\nabla_0f\|^{\tp-2}\ip{\nabla_0f}{\nabla_0\phi}.
\end{eqnarray*}  
Let $\mathcal{L}_{2n+1}$ and $\mathcal{H}$ be the measures from \eqref{coarea}. Applying Stokes' Theorem,
\begin{eqnarray*}
\lefteqn{ \int_{\mathcal{A}}\|\nabla_0f\|^{\tp-2} \ip{\nabla_0f}{\nabla_0\phi}d\mathcal{L}_{2n+1} = 
\int_{\mathcal{A}}\div(\phi \|\nabla_0f\|^{\tp-2}\nabla_0f)d\mathcal{L}_{2n+1}} \\
 &=& \int_{\mathcal{A}} \sum_{l=1}^{2n} X_l[\phi \|\nabla_0f\|^{\tp-2}X_lf] d\mathcal{L}_n \\ 
 &=& \int_{\mathcal{A}} \sum_{l=1}^{2n} \frac{\partial}{\partial x_l}[\phi \|\nabla_0f\|^{\tp-2}X_lf]
+ \sum_{i=1}^n 2kc(x_{i+n} -a_{i+n})\Sigma^{k-1}\frac{\partial}{\partial t}[ \phi \|\nabla_0f\|^{\tp-2}X_if] \\
& & \mbox{}- \sum_{j=n+1}^{2n} 2kc(x_{j-n} -a_{j-n})\Sigma^{k-1}\frac{\partial}{\partial t}[ \phi \|\nabla_0f\|^{\tp-2}X_jf]d\mathcal{L}_n \\ &=&
\int_{\mathcal{A}} \div_{\eucl}
[\xi] d\mathcal{L}_n 
\end{eqnarray*}
where the $(2n+1)$-vector $\xi$ is defined by 
\begin{eqnarray*}
\xi=\left[\begin{array}{c}\phi \|\nabla_0f\|^{\tp-2}X_1f \\
\phi \|\nabla_0f\|^{\tp-2}X_2f \\
\ldots \\ 
\phi \|\nabla_0f\|^{\tp-2}X_{2n}f\\
2kc\Sigma^{k-1} \phi \|\nabla_0f\|^{\tp-2}\bigg(\D\sum_{i=1}^n (x_{i+n} -a_{i+n})X_if- \D\sum_{j=n+1}^{2n} (x_{j-n} -a_{j-n}) X_jf\bigg)
\end{array}\right]
\end{eqnarray*}

Thus, 
\begin{eqnarray*}
\lefteqn{ \int_{\mathcal{A}}\|\nabla_0f\|^{\tp-2} \ip{\nabla_0f}{\nabla_0\phi}d\mathcal{L}_{2n+1}  =  
\int_{\partial\mathcal{A}} \frac{1}{\|\nu\|} \sum_{l=1}^{2n} \phi \|\nabla_0f\|^{\tp-2}X_lf \nu_l
+  \xi_{2n+1}\nu_{2n+1}  d\mathcal{H}} \\ 
& = & = -\int_{\partial B_r} \frac{1}{\|\nu\|}  \sum_{l=1}^{2n} \phi \|\nabla_0f\|^{\tp-2}X_lf \nu_l
+  \xi_{2n+1}\nu_{2n+1}  d\mathcal{H} 
\end{eqnarray*}
where $\nu$ is the outward Euclidean normal.  Recalling that
$$ \psi(x_1, x_2, \ldots, x_{2n}, t) = [h(x_1, x_2, \ldots, x_{2n}, t)]^\frac{1}{4k},$$
we proceed with the computation,
\begin{eqnarray*}
\lefteqn{ \int_{\mathcal{A}}\|\nabla_0f\|^{\tp-2} \ip{\nabla_0f}{\nabla_0\phi}d\mathcal{L}_n } \\ 
& = & -\int_{\partial B_r} \frac{1}{\|\nu\|}\bigg(\sum_{l=1}^{2n} \phi \|\nabla_0f\|^{\tp-2}X_lf \nu_l
+  \xi_{2n+1}\nu_{2n+1}\bigg)  d\mathcal{H} \\ 
& = & -\int_{\partial B_r} \frac{\alpha \psi^{\alpha-1}}{\|\nu\|} \phi \|\nabla_0\psi^\alpha\|^{\tp-2}
\Bigg(\sum_{l=1}^{2n}X_l\psi\frac{\partial\psi}{\partial x_l}+2kc\Sigma^{k-1} \\
& & \times\bigg(\D\sum_{i=1}^n (x_{i+n} -a_{i+n})X_i\psi- \D\sum_{j=n+1}^{2n} (x_{j-n} -a_{j-n}) X_j\psi\bigg)
\frac{\partial\psi}{\partial t}\Bigg) d\mathcal{H} \\
 & = & -\int_{\partial B_r} \frac{\alpha \psi^{\alpha-1}}{\|\nu\|} \phi 
\|\nabla_0\psi\|^{\tp-2}|\alpha|^{\tp-2} \psi^{(\tp-2)(\alpha-1)}
\bigg( \| \nabla_0\psi\|^2 \bigg) d\mathcal{H} \\ & = &
-\int_{\partial B_r} \frac{|\alpha|^{\tp-2}\alpha \psi^{(\tp-1)(\alpha-1)}}{\|\nu\|} \phi 
\|\nabla_0\psi\|^\tp d\mathcal{H}.
\end{eqnarray*}
Recall that by definition, $\psi \equiv r$ on $\partial B_r$.  We then have
\begin{eqnarray*}
\int_{\mathcal{A}}\|\nabla_0f\|^{\tp-2} \ip{\nabla_0f}{\nabla_0\phi} d\mathcal{L}_n &=&
- |\alpha|^{\tp-2} \alpha r^{1-Q} 
\int_{\partial B_r} \frac{\phi \|\nabla_0\psi\|^\tp}{\|\nu\|} \, d\mathcal{H}. 
\end{eqnarray*} 
Letting $r \to 0$, we apply \eqref{density} and obtain
\begin{equation} \label{scaling}
\int_{\mathcal{A}}\|\nabla_0f\|^{\tp-2} \ip{\nabla_0f}{\nabla_0\phi} d\mathcal{L}_{2n+1} \; \to \;
- |\alpha|^{\tp-2} \alpha (Q\sigma_\tp) \phi(x_0).
\end{equation}
We then obtain the case for $\tp \neq Q$.  
The case of $\tp=Q$ is similar and left to the reader.
\end{proof}

We have the following corollary.   
\begin{corollary}
The function $\psi$, as defined above, is infinite harmonic in the space 
$G_n \setminus \{ x_0 \}$.
\end{corollary}
\begin{proof}
We use the formula that for a smooth function $u$,
$$
2\Delta_{\infty}u=\nabla_0u \cdot \nabla_0\|\nabla_0u\|^2.
$$ 
Computing as in the Theorem, we have
$$
\|\nabla_0\psi\|^2 = c^2 \Sigma^{2k-1}h^{\frac{1-2k}{2k}}.
$$
Thus we obtain for $i=1$ to $n$,
\begin{equation*}
 X_i\|\nabla_0\psi\|^2 = 2c^2\Sigma^{2k-2}h^{\frac{1-4k}{2k}}
(2k-1)\bigg(h(x_i-a_i)-c^2(x_i-a_i)\Sigma^{2k}-c\Sigma^{k}(x_{i+n}-a_{i+n})(t-s)\bigg)
\end{equation*}
and for $j=n+1$ to $2n$,
\begin{equation*}
 X_j\|\nabla_0\psi\|^2  =  2c^2\Sigma^{2k-2}h^{\frac{1-4k}{2k}}
(2k-1)\bigg(h(x_j-a_j)-c^2(x_j-a_j)\Sigma^{2k}+c\Sigma^{k}(x_{j-n}-a_{j-n})(t-s)\bigg)
\end{equation*}
Calculations then give us
\begin{eqnarray*}
\Delta_{\infty}\psi & = & c^2\Sigma^{2k-2}h^{\frac{1-4k}{2k}}
(2k-1) \bigg( c^2h\Sigma^{k+1}-c^4\Sigma^{3k+1}-c^2\Sigma^{k+1}(t-s)^2\bigg) \\
 & = & 0.
\end{eqnarray*} 
The corollary then follows.
\end{proof}

\section{Spherical Capacity}

In this section, we will use previous results to compute the capacity of spherical rings centered at the point $x_0=(a_1, a_2, \ldots, a_{2n}, s)$.
We first recall the definition of $\tp$-capacity.
\begin{definition}
Let $\Omega \subset G_n$ be a bounded, open set, and $K \subset \Omega$ 
a compact subset. For $1 \leq \tp < \infty$ we define the $P$-capacity as
$$
\cap_\tp(K, \Omega) = 
\inf \left\{ \int_\Omega \|\nabla_0u\|^\tp : 
u \in C^\infty_0(\Omega), \; u|_K = 1 \right\}.
$$
\end{definition}

We note that although the definition is valid for $\tp=1$, we will consider only $1 < \tp < \infty$, as in the previous sections. Because $\tp$-harmonic functions are minimizers to the energy integral
$$
\int_{G_n} \|\nabla_0f\|^\tp
$$
it is natural to consider $\tp$-harmonic functions when computing the capacity.  In particular, an easy calculation similar to the previous section shows
$$
u(x) = \left\{\begin{array}{lcc}
\frac{\D\psi(x)^\alpha - \D R^\alpha}{\D r^\alpha - \D R^\alpha} & 
\textrm{when} & 
\tp \neq Q \\
\\
\frac{\D\log{\psi(x)} - \D\log{R}}{\D\log{r} - \D\log{R}} & 
\textrm{when} & 
\tp = Q
\end{array}\right.
$$
is a smooth solution to the Dirichlet problem
$$
\left\{\begin{array}{ccl}
\Delta_\tp u = 0 & \textrm{in} & B(x_0, R) \setminus B(x_0, r) \\
u = 1 & \textrm{on} & \partial B(x_0, r) \\
u = 0 & \textrm{on} & \partial B(x_0, R)
\end{array}\right.
$$
for $1< \tp < \infty$. 

We state the following theorem, which follows from the computations of 
the previous section.
\begin{thm}
Let $0 < r < R$ and $1 < \tp < \infty$.  Then we have
$$
\cap_\tp\big(B(x_0,r), B(x_0, R)\big) = \left\{\begin{array}{lcc}
\alpha^{\tp-1} Q \sigma_\tp \big( r^\alpha - R^\alpha \big)^{1-\tp} & 
\textrm{when} & 1 < \tp < Q \\
\\
Q\sigma_Q [\log{R} - \log{r}]^{1-Q} & \textrm{when} & 
\tp = Q \\
\\
\alpha^{\tp-1} Q \sigma_\tp \big( R^\alpha - r^\alpha \big)^{1-\tp} & 
\textrm{when} & \tp > Q.
\end{array}\right.
$$
\end{thm}

\end{document}